\documentclass{amsart}
\usepackage{amsmath,amssymb,amsthm, enumerate, breqn}

\numberwithin{equation}{section}

\theoremstyle{plain}
\newtheorem{theorem}{Theorem}[section]

\newtheorem{proposition}[theorem]{Proposition}

\theoremstyle{definition}
\newtheorem{definition}[theorem]{Definition}

\theoremstyle{remark}
\newtheorem{remark}[theorem]{Remark}
\newtheorem*{acknowledgments}{Acknowledgments}

\DeclareMathOperator{\Km}{Km}
\DeclareMathOperator{\Hom}{Hom}

\DeclareMathOperator{\II}{II}

\DeclareMathOperator{\Ino}{Ino}

\begin{document}

\title[The Mordell-Weil lattice of an Inose surface]{The Mordell-Weil
  lattice of an Inose surface arising from isogenous elliptic curves}

\author{Kazuki Utsumi}
\address{College of Science and Engineering, Ritsumeikan
  University \endgraf 1-1-1 Noji-higashi, Kusatsu Shiga 525-8577 Japan}
\email{kutsumi@fc.ritsumei.ac.jp}


\begin{abstract}
  An elliptic $K3$ surface with two $\II^{*}$ fibers is called the
  Inose surface. In this paper, we give a method to find a section of
  an Inose surface corresponding to an isogeny of general degree
  between two elliptic curves. In particular, we show examples of
  bases of the Mordell-Weil lattices of Inose surfaces arising from
  isogenies of degrees $5$ and $6$.
\end{abstract}

\maketitle
 
\section{Introduction}

Elliptic $K3$ surfaces with large Picard numbers play an important
role in the study of the geometry, arithmetic and moduli of $K3$
surfaces. Shioda and Inose \cite{Shioda-Inose} classified singular
$K3$ surfaces, that is, complex $K3$ surfaces with maximum Picard
number $20$. They showed any elliptic $K3$ surface with two $\II^{*}$
fibers can be constructed as a double cover of the Kummer surface
$\Km(E_1 \times E_2)$ with the product of two elliptic curves $E_1$
and $E_2$. We denote this $K3$ surface by $F^{(1)}_{E_1, E_2}$ (or
$F^{(1)}$ for short). The Kodaira-N\'{e}on model of
$F^{(1)}_{E_1, E_2}$ is called the Inose surface associated with $E_1$
and $E_2$, and denoted by $\Ino(E_1, E_2)$. The notation of $F^{(1)}$
reflects that it is a part of the construction of elliptic $K3$
surfaces of high rank by Kuwata \cite{Kuwata:MW-rank}, where he
constructed $F^{(1)}, \ldots, F^{(6)}$ which has various Mordell-Weil
rank up to $18$. Their precise definitions and equations are given in
Section \ref{sec:inose}.

The Mordell-Weil lattice of $F^{(1)}$ has been known to be isomorphic
to the lattice $\Hom(E_1, E_2)\langle 2\rangle$ if $E_1$ and $E_2$ are
nonisomorphic (see Section \ref{sec:MWL}). Not much is known about the
coordinates of the section corresponding to a given
$\varphi \in \Hom(E_1, E_2)$ and the Weierstrass equation of $F^{(1)}$
except for $\deg \varphi =2$. In the case where $\deg \varphi =2$, the
calculations are straight forward. Recently, Kumar and Kuwata gave an
example of the case $\deg \varphi =4$ in \cite[Example
9.2]{Kumar-Kuwata}. Moreover, Kuwata and the author
\cite{Kuwata-Utsumi} wrote down a formula of the section of $F^{(1)}$
in the case where $\deg \varphi=3$ by modifying their method. They
first worked with the surface $F^{(6)}$, which has a simple affine
model that can be viewed as a cubic curve (named $C_u$ in
(\ref{eq:Cu})) with a rational point. Their key idea is finding
a conic curve passing certain $5$ points on the cubic curve, which
gives the sixth point of the intersection between the conic and cubic
curves. We generalize the relation between the cubic and conic curves
to the one between the cubic and of degree
$\lceil (\deg \varphi) /2\rceil $ curves by using the Cayley-Bacharach
theorem which is stated in Section \ref{sec:CBT}.

In this paper, we extend their method in \cite{Kumar-Kuwata,
  Kuwata-Utsumi} to write down the section of $F^{(1)}$
corresponding to an isogeny of general degree. We explain it in detail
in Section \ref{sec:rat}. The core of our method is Theorem
\ref{thm:passing}, which gives us two sections of $F^{(6)}$
lifting the section of $F^{(1)}$.

Finally, we study examples in Sections \ref{sec:iso5} and
\ref{sec:iso6}. We actually calculate the section of
$F^{(1)}$ corresponding to isogenies of degrees $5$ and $6$ according
to the method in Section \ref{sec:rat}, both of which form bases of
$F^{(1)}$.

\section{Inose surface}\label{sec:inose}

Throughout this paper the base field $k$ is
assumed to be a number field.

Let $\Km(E_1 \times E_2)$ be the Kummer surface associated with the
product of elliptic curves $E_1$ and $E_2$, that is, the minimal
resolution of the quotient surface $E_1 \times E_2/ \{\pm 1\}$. If the
two elliptic curves are defined by the equations
\begin{equation}\label{eq:E1E2}
  E_1 : y_1^2 = x_1^3+a_2x_1^2+a_4x_1+a_6, \quad E_2 : y_2^2 = x_2^3+a'_2x_2^2+a'_4x_2+a'_6,
\end{equation}
an affine singular model of $\Km(E_1 \times E_2)$ is given as the
hypersurface in $\mathbb{A}^3$ defined by the equation
\begin{equation}
  \label{eq:singularmodel}
  x_2^3+a'_2x_2^2+a'_4x_2+a'_6 = t^2 \left(x_1^3+a_2x_1^2+a_4x_1+a_6 \right),
\end{equation}
where $t=y_2/y_1$. Take a parameter $t=u^3$, and consider
(\ref{eq:singularmodel}) as the following cubic curve $C_u$ in
$\mathbb{P}^2 = \{(x_1 : x_2 : z)\}$ over $k(u)$.
\begin{equation}
  \label{eq:Cu}
  C_u : x_2^3 + a'_2x_2^2z+a'_4x_2z^2+a'_6z^3 = u^6 (x_1^3+a_2x_1^2z+a_4x_1z^2+a_6z^3)
\end{equation}
Then, this curve has a rational point $(x_1:x_2:z) = (1:
u^2:0)$. Using this point as the origin, we consider it as an elliptic
curve over $k(u)$, and convert it to the Weierstrass form
\begin{equation}
  \label{eq:WEQ}
 F^{(6)}_{E_1, E_2} :  Y^2 = X^3-\frac{1}{3}AX + \frac{1}{64}\left( \Delta_{E_1} u^6 + B+\frac{\Delta_{E_2}}{u^6}\right),
\end{equation}
where
\begin{equation}
  \begin{cases}
    A = (a_2^2-3a_4)({a_2'}^2-3a_4'),\\
    B=\frac{32}{27}\left( 2a_2^3-9a_2a_4+27a_6\right)\left(2{a'_2}^3-9a'_2a'_4+27a'_6\right),\\
    \Delta_{E_1}=16\left(a_2^2a_4^2-4a_2^3a_6+18a_2a_4a_6-4a_3-27a_6^2\right),\\
    \Delta_{E_2}=16\left({a'_2}^2{a'_4}^2-4{a'_2}^3{a'_6}+18{a'_2}{a'_4}{a'_6}-4{a'}_3-27{a'_6}^2\right).
  \end{cases}
\end{equation}
The change of coordinates between (\ref{eq:Cu}) and (\ref{eq:WEQ}) are given by
\begin{equation}
  \label{eq:trans}
  \begin{cases}
    X= \dfrac{c_6 u^6 + c_4u^4+c_2u^2+c_0}{3u^2\left( \left(3x_1+a_2z\right) u^2 -\left(3x_2+a'_2z\right) \right)}, \\ 
    Y= \dfrac{d_{10}u^{10}+d_6u^6+d_5u^4+d_0}{2u^3\left(\left(3x_1+a_2z\right) u^2 -\left(3x_2+a'_2z\right) \right)^2},
  \end{cases}
\end{equation}
where
\begin{equation*}
  \begin{aligned}
    c_6=& 6 (a_{2}^{2}-3 a_{4})x_1+3 (a_{2} a_{4}-9 a_{6}) z,\\
    c_4=& (a_{2}^{2}-3 a_{4}) (3x_2+a'_{2} z ), \\
    c_2=& -({a'_{2}}^{2}-3 a'_{4}) (3x_1+ a_{2}z),\\
    c_0=& -6({a'_{2}}^{2}-3a'_{4})x_2-3 (a'_{2}a'_{4}-9 a'_{6}) z, \\
    d_{10}=& -3(2 a_{2}^{3}-9 a_{2} a_{4}+27 a_{6}) x_{1}^{2}
    + 2 (a_{2}^{4}-9 a_{2}^{2} a_{4}-27 a_{2} a_{6}+27 a_{4}^{2})x_1 z\\
    & +(a_{2}^{3} a_{4}-27 a_{2}^{2} a_{6}+54 a_{4} a_{6})z^2, \\
    d_{6} =& -6 (a_{2}^{2}-3 a_{4}) ({a'_{2}}^{2}-3 a'_{4})  x_{1}z +3 (2 a_{2}^{3}-9 a_{2} a_{4}+27 a_{6}) x_{2}^{2}\\
    &+2 a'_{2} (2 a_{2}^{3}-9 a_{2} a_{4}+27 a_{6})  x_{2}z
    +(2 a_{2}^{3} a'_{4}-3 a_{2} a_{4} {a'_{2}}^{2}+27 a_{6} {a'_{2}}^{2}-54 a_{6} a'_{4}) z^{2},\\
    d_{4}=& 3 (2{a'_{2}}^{3}-9a'_{2} a'_{4}+27 a'_{6}) x_{1}^{2}+2 a_{2} (2 {a'_{2}}^{3}-9 a'_{2} a'_{4}+27 a'_{6})  x_{1}z\\
    & -6 (a_{2}^{2}-3 a_{4}) ({a'_{2}}^{2}-3a'_{4})x_{2}z
    -(3 a_{2}^{2} a'_{2} a'_{4}-2 a_{4} {a'_{2}}^{3}-27 {a_{2}}^{2} a'_{6}+54 a_{4} a'_{6}) z^{2},\\
    d_{0}= & -3 (2 {a'_{2}}^{3}-9 a'_{2} a'_{4}+27 a'_{6}) x_{2}^{2}
    +2 ({a'_{2}}^{4}-9 {a'_{2}}^{2} a'_{4}-27 a'_{2} a'_{6}+27 {a'_{4}}^{2}) x_{2}z\\
    & +({a'_{2}}^{3} a'_{4}-27{a'_{2}}^{2}a'_6+54 a'_{4} a'_{6}) z^{2}.
  \end{aligned}
\end{equation*}

\begin{remark}
  The origin $O=(1: u^2:0)$ is not an inflection point of the cubic
  $C_u$. Thus, three colinear points $P, Q, R \in C_u$ do not satisfy
  the equation $P+Q+R=O$ under the group law. Instead, we have
  $P+Q+R=\bar{O}$, where $\bar{O}$ is the third point of intersection
  between $C_u$ and the tangent line at $O$. 
\end{remark}

Let $s=t^2=u^6$. Define $F^{(1)}_{E_1, E_2}$ to be
\begin{equation}
  \label{eq:F1}
  F^{(1)}_{E_1, E_2} : Y^2 = X^3 - \frac{1}{3} AX +\frac{1}{64}\left( \Delta_{E_1} s + B + \frac{\Delta_{E_2}}{s}\right).
\end{equation}
This elliptic fibration has two reducible fibers of type $\II^*$ at
$s=0$ and $s=\infty$. The Kodaira-N\'{e}ron model of
$F^{(1)}_{E_1, E_2}$ is called the Inose surface associated with $E_1$
and $E_2$, and it is denoted by $\Ino(E_1, E_2)$.

\begin{definition}
  For $n \geq 1$, the elliptic surface $F^{(n)}_{E_1, E_2}$ (or
  $F^{(n)}$ for short) over $k$ is defined by
  \begin{equation}
    \label{eq:Fn}
    F^{(n)}_{E_1 E_2} : Y^2 = X^3 -\frac{1}{3}A X + \frac{1}{64}\left(\Delta_{E_1} s^n + B + \frac{\Delta_{E_2}}{s^n}\right).
  \end{equation}
\end{definition}

\begin{remark}
  (1) The Kodaira-N\'{e}ron model of $F^{(n)}$ is a $K3$
  surface for $n=1, \ldots, 6$, but for
  $n \geq 7$(\cite{Kuwata:MW-rank}).

  (2) The map $\Km(E_1, E_2) \to \mathbb{P}^1$ induced by
  $(x_1, x_2, t) \mapsto t$ in (\ref{eq:singularmodel}) is an elliptic
  fibration. Since $u^6=t^2$, this elliptic fibration is isomorphic to
  $F^{(2)}$. However, the isomorphism between
  (\ref{eq:singularmodel}) and (\ref{eq:Fn}) for $n=2$ may not be
  defined over $k$ itself. It is defined over an extension of $k$
  including some of the $2$-torsion points of $E_1$ and $E_2$.

  (3) The elliptic fibration $F^{(2)}$ is sometimes called Inose's
  pencil or Inose('s) fibration (cf. \cite{Shioda:correspondence},
  \cite{Kuwata-Shioda}, \cite{Kumar-Kuwata}). However, these names are
  recently also used for $F^{(1)}$ (cf. \cite{SS:ES}, \cite{SS:MWL}).
\end{remark}

\section{Mordell-Weil lattice of $F^{(1)}$}\label{sec:MWL}

Our goal is to write down an explicit section of
$F^{(1)}$ arising from an isogeny. In the case where $E_1$ and $E_2$
are not isogenous, the Mordell-Weil lattice $F^{(1)}(\bar{k}(s))$ is
trivial, and we have nothing to do. We are interested in the case
where $E_1$ and $E_2$ are isogenous and but not isomorphic over
$\bar{k}$ because of the following.

\begin{proposition}[{\cite[Theorem 6.3]{Shioda:correspondence}}]\label{thm:MWL}
  Let $E_1$ and $E_2$ be two elliptic curves not isomorphic to each
  other over $\bar{k}$. Then, the Mordell-Weil group
  $F^{(1)}(\bar{k}(s))$ is torsion free, and isomorphic to the lattice
  $\Hom_{\bar{k}}(E_1,E_2)\langle 2 \rangle$, where the pairing of
  $\Hom_{\bar{k}}(E_1, E_2)$ is given by
  \[
    (\varphi, \psi) = \frac{1}{2}\left( \deg(\varphi+\psi) - \deg \varphi - \deg \psi\right)
    \quad \varphi, \psi \in \Hom_{\bar{k}}(E_1, E_2).
  \]
  The notation $\langle n \rangle$ means that the pairing of the
  lattice is multiplied by $n$.
\end{proposition}

For a given $\varphi \in \Hom_{k}(E_1, E_2)$, we would like to compute the
section of $F^{(1)}$ corresponding to $\varphi$
explicitly. To do so, we consider the inclusion
\[
  \Hom_{\bar{k}}(E_1, E_2)\langle 2\rangle \simeq F^{(1)}(\bar{k}(s))
  \hookrightarrow \Hom_{\bar{k}}(E_1,E_2)\langle 12 \rangle \subset
  F^{(6)}(\bar{k}(u))
\]
induced by $s \mapsto u^6$, and we look for a section in
$F^{(6)}(k(u))$, which is converted from a rational point
on $C_u$ by (\ref{eq:trans}). We can obtain this rational point from
the Cayley-Bacharach theorem.

\section{The Cayley-Bacharach theorem}\label{sec:CBT}

In this section, we introduce the Cayley-Bacharach theorem. See
\cite{CBT} for details of this theorem.

\begin{theorem}[Cayley-Bacharach theorem {\cite[Theorem CB4]{CBT} }]\label{thm:CB} Let
  $X_1, X_2 \subset \mathbb{P}^2$ be plane curves of degree $d$ and
  $e$ respectively, meeting in a collection of $d \cdot e$ distinct
  points $\Gamma=\{p_1, \ldots, p_{de}\}$. If
  $C \subset \mathbb{P}^2$ is any plane curve of degree $d+e-3$
  containing all but one point of $\Gamma$, then $C$ contains all of
  $\Gamma$.
\end{theorem}

\begin{remark}
  It is not actually necessary that $X_1$ and $X_2$ intersect in
  distinct points. For example, if $P \in X_1 \cap X_2$ is a point of
  multiplicity $2$, then one needs to require that $X_1, X_2$ and $C$
  have the same tangent direction at $P$.
\end{remark}

Applying Theorem \ref{thm:CB} with $d=3$, we have the following.

\begin{proposition}\label{prop:CB}
  Let $X \subset \mathbb{P}^2$ be a smooth cubic plane curve, and
  $\Gamma$ be a collection of distinct $3e-1$ points on $X$. Then,
  there exists the $3e$-th point $P$ on $X$ such that any plane curve
  $C \subset \mathbb{P}^2$ of degree $e$ containing all of
  $\Gamma$ passes through $P$.
\end{proposition}

\section{A section of $F^{(1)}$ corresponding to an isogeny}\label{sec:rat}
In this section, we assume that there exists an isogeny
$\varphi: E_1 \to E_2$ of degree $d \geq 2$ over $k$, and
$j(E_1) \neq j(E_2)$. We find an explicit section of $F^{(1)}$
corresponding to $\varphi$.

Suppose that $E_1$ and $E_2$ are given by (\ref{eq:E1E2}) . Then,
$\varphi$ can be written in the form
\begin{equation}
  \label{eq:phi}
  \varphi : (x_1, y_1) \mapsto (x_2, y_2) = \left( \varphi_x(x_1), \varphi_y(x_1)y_1\right).
\end{equation}
We work with the cubic curve $C_u$ over $k(u)$ in
$\mathbb{P}^2 = \{(x_1: x_2: z)\}$ given by (\ref{eq:Cu}), which is
isomorphic over $k(u)$ to $F^{(6)}$ with the choice of origin
$O=(1 : u^2: 0)$. Consider the curve of degree $d$ given by
$x_2=\varphi_x(x_1)$. The intersection of these two curves (and $z=1$)
\begin{equation}
  \begin{cases}
    x_2^3+a'_2x_2^2+a'_4x_2+a'_6 = u^6\left( x_1^3+a_2 x_1^2+a_4x_1+a_6\right)\\
    x_2 =\varphi_x(x_1)
  \end{cases}
\end{equation}
gives a divisor of degree $3d$ in $C_u$. Since we have
$\varphi_x(x_1)^3+a'_2\varphi_x(x_1)^2 + a'_4\varphi_x(x_1) + a'_6 =
\varphi_y(x_1)^2 y_1^2 = \varphi_y(x_1)^2
\left(x_1^3+a_2x_1^2+a_4x_1+a_6\right)$, the first equation reduces to
\begin{equation}
  \label{eq:div3d}
  \left( \varphi_y(x_1) - u^3\right) \left(\varphi_y(x_1)+u^3\right)
  \left(x_1^3+a_2x_1^2+a_4x_1+a_6\right)=0.
\end{equation}
Let $p^{\pm}(x_1)$ be the numerator of
$\varphi_y(x_1)- \left( \pm u^3\right)$ and $r = \deg
p^{\pm}(x_1)$. Then, $p^{\pm}(x_1)$ defines the divisor
$D^{\pm}_{\varphi} = Q^{\pm}_1 + \cdots + Q^{\pm}_r$ on $C_u$, where
the $Q^{\pm}_i$ are the $\overline{k(u)}$-rational points on $C_u$ of
the form $(x_1: \varphi_x(x_1): 1)$ whose $x_1$ coordinates are the $r$
roots of $p^{\pm}(x_1)=0$.

\begin{theorem}\label{thm:passing}
  \begin{enumerate}[(i)]
  \item If $d=\deg \varphi$ is odd, then
    $r =\deg p^{\pm}(x_1)=(3d-3)/2$, and there exists a plane curve
    $C^{+}$ (resp. $C^{-}$) of degree $(d+1)/2$ passing through $r+2$
    points $Q^{+}_1, \ldots, Q^{+}_r$ (resp.
    $Q^{-}_1, \ldots, Q^{-}_{r}$) , $O$ and $\bar{O}$.

  \item If $d$ is even, then $r=(3d-2)/2$, and there exists a plane
    curve $C^{+}$ (resp. $C^{-}$) of degree $d/2$ passing through $r$
    points $Q^{+}_1, \ldots, Q^{+}_r$ (resp.
    $Q^{-}_1, \ldots, Q^{-}_r$).

  \item For both even $d=2l$ and odd $d=2l-1$, any curve $C^{+}$
    (resp. $C^{-}$) in (i) or (ii) passes through $3l$-th point
    $Q^{+}$ (resp. $Q^{-}$) on $C_u$, which is a $k(u)$-rational
    point.
  \end{enumerate}
\end{theorem}

\begin{proof} 
  (i) Let $d=2l-1$. Since the denominator of $\varphi_y(x_1)$ and
  $x_1^2+a_2x_1^2+a_4x_1+a_6$ are relatively prime, we have
  $r = (3d-3)/2 =3l-3$. We look for a plane curve $C^{+}$ of degree
  $(d+1)/2=l$ over $k(u)$ given by
  \begin{equation}
    C^{+} : q^{+}(x_1, x_2, z) = c_1 x_1^l + c_2 x_1^{l-1}x_2 + \cdots +c_{n} z^l
  \end{equation}
  which passes through $Q^{+}_1, \ldots, Q^{+}_r, O$ and
  $\bar{O}=(\alpha: \beta: \gamma)$, where $n = (l+1)(l+2)/2$. The curve
  $C^{+}$ passes through $Q^{+}_1, \ldots, Q^{+}_r$ if and only
  $q^{+}(x_1, \varphi_x(x_1), 1)$ is divisible by $p^{+}(x_1)$. Adding
  two equations $q^{+}(1, u^2, 0) = q^{+}(\alpha, \beta, \gamma)=0$ which
  mean that $C^{+}$ passes through $O$ and $\bar{O}$, we have a
  system of $r+2$ homogeneous linear equations in $c_1, \ldots,
  c_n$. Since $n > r+2 $, there is a non-trivial solution of this
  system, which gives us the coefficients of the curve
  $C^{+}$. Similarly, we obtain a curve $C^{-}$ from the points
  $Q^{-}_1, \ldots, Q^{-}_r, O, \bar{O}$ and the polynomial
  $p^{-}(x_1)$.
  
  (ii) Let $d=2l$. Since cancellation occurs between the denominator
  of $\varphi_y(x_1)$ and $x_1^3+a_2x_1^2+a_4x_1+a_6$ at the $x_1$
  coordinate of one of the $2$-torsion points of $E_1$, we have
  $r=(3d-2)/2=3l-1$. We can find a plane curve $C^{\pm}$ of degree
  $d/2 = l$ passing through $Q^{\pm}_1, \ldots, Q^{\pm}_r$ in a
  similar way in (i).

  (iii) Any $C^{\pm}$ in (i) or (ii) contains $3l-1$ points on
  $C_u$. Thus, there exists a $3l$-th point $Q^{\pm}$ on $C_u$ by
  Proposition \ref{prop:CB}. Since
  $D^{+}_{\varphi}, D^{-}_{\varphi}, O$ and $\bar{O}$ are all defined
  over $k(u)$, $Q^{+}$ and $Q^{-}$ are $k(u)$-rational points on
  $C_u$.
\end{proof}

The $k(u)$-rational points $Q^{+}$ and $Q^{-}$ define the section
$P_{\varphi}$ of $F^{(1)}$ for which we are looking.

\begin{theorem}\label{thm:section-F1}
  Let $Q^{+}$ and $Q^{-}$ be the points on $C_u$ in Theorem
  \ref{thm:passing}. Let $\Psi: C_u \to F^{(6)}$ be the isomorphism
  over $k(u)$ defined by the formula (\ref{eq:trans}), and
  $P^{+}_{\varphi}$ (resp. $P^{-}_{\varphi}$) be the point in
  $F^{(6)}(k(u))$ given by $ \Psi(Q^{+})$ (resp. $\Psi(Q^{-})$). Then,
  $P^{+}_{\varphi} - P^{-}_{\varphi}$ is the image of
  $F^{(1)}(k(s)) \to F^{(6)}(k(u))$ induced by $s \mapsto u^6$. The
  height of its pre-image $P_{\varphi}$ in $F^{(1)}(k(s))$ is $2d$.
\end{theorem}

\begin{proof}
  This theorem is essentially same as \cite[Proposition
  5.3]{Kuwata-Utsumi}, and can be proven in the same way.
\end{proof}

We show examples in the cases where $\deg \varphi =5,6$ in the next
sections.

\section{Example: Isogeny of degree $5$}\label{sec:iso5}

In this section, we give an example of the section of $F^{(1)}$
arising from an isogeny of degree $5$.

Let $E_1$ and $E_2$ be elliptic curves over $k=\mathbb{Q}$ given by
\begin{equation}
  \label{eq:E1E2-5}
  E_1 : y_1^2 = x_1^3-4x_1^2+16, \quad E_2: y_2^2 = x_2^3-4x_2^2-160x_2-1264.
\end{equation}
The curves $E_1$ and $E_2$ are labeled by 11a.3 and 11a.2 respectively
in LMFDB\cite{LMFDB}. The map
$\varphi: E_1 \to E_2 : (x_1, x_2) \mapsto (\varphi_x(x_1),
\varphi_y(x_1)y_1)$ is a $5$-isogeny, where
\begin{equation}
  \label{eq:iso5}
  \begin{aligned}
    \varphi_x(x_1) &= \frac{x_{1}^{5}-8 x_{1}^{4}+48 x_{1}^{3}-512 x_{1}+1024}
    {x_{1}^{2} \left(x_1-4\right)^2},\\
    \varphi_y(x_1) &=\frac{\left(x_{1}^{3}+4 x_{1}^{2}+16
        x_{1}-64\right) \left(x_{1}^{3}-16 x_{1}^{2}+64
        x_{1}-128\right)}{x_{1}^{3} \left(x_{1}-4\right)^{3}}.
\end{aligned}
\end{equation}
In this case, the plane cubic curve $C_u$ is given by
\begin{equation}
  \label{eq:Cu5}
  C_u: x_2^3-4x_2^2z-160x_2z^2-1264z^3 = u^6\left( x_1^3-4x_1^2z+16z^3\right),
\end{equation}
which converts to 
\begin{equation}
  \label{eq:F6-5}
 F^{(6)} : Y^2= X^3 - \frac{7936}{3}X - 704 u^6 -\frac{6082432}{27} - \frac{10307264}{u^6}
\end{equation}
with the choice of origin $O=(1: u^2:0)$ by (\ref{eq:trans}). The
elliptic surface $F^{(1)}$ is given by $s=u^6$ in (\ref{eq:F6-5}).
\begin{equation}
  \label{eq:F1-5}
 F^{(1)} : Y^2= X^3 - \frac{7936}{3}X - 704 s -\frac{6082432}{27} - \frac{10307264}{s}
\end{equation}

We would like to compute the $k(u)$-rational point $P_{\varphi}$ in
$F^{(1)}$ in Theorem \ref{thm:section-F1}. Substituting
$x_2 =\varphi_x(x_1)$ (and $z=1$) into (\ref{eq:Cu5}), we obtain
\begin{equation}
  \frac{p^{+}(x_1) p^{-}(x_1) \left(x_1^3-4x_1+16\right)}{x_1^6\left(x_1-4\right)^6}=0, 
\end{equation}
where
\begin{equation}
  \label{eq:p5}
  \begin{aligned}
    p^{+}(x_1) =& \left(1-u^3\right)x_1^6 -12\left(1-u^3\right)x_1^5
    +16\left(1-3u^3\right)x_1^4-64\left(3-u^3\right)x_1^3\\
    & +1536x_1^2- 6144x_1+8192, \\
    p^{-}(x_1) =& \left(1+u^3\right)x_1^6 -12\left(1+u^3\right)x_1^5
    +16\left(1+3u^3\right)x_1^4-64\left(3+u^3\right)x_1^3\\
    &+1536x_1^2- 6144x_1+8192. 
  \end{aligned}
\end{equation}
Let $Q^{+}_1, \ldots, Q^{+}_6$ be the points on $C_u$ of the form
$(x_1: \varphi_x(x_1): 1)$ whose $x_1$ coordinates are the $6$ roots
of $p^{+}(x_1)=0$.  Applying Theorem \ref{thm:passing} with $d=5$, we
can find a plane cubic curve $C^{+}$ passing through the $8$ points
$Q^{+}_1, \ldots, Q^{+}_6, O$ and $\bar{O}$ by the method shown in the
proof of this theorem. Note that the point $\bar{O}$ is in this case
given by
\begin{equation}
  \bar{O} = \left(31u^6-372u^2+2501 : u^2(19u^6+12u^4+2129) :  9u^2(u^4-31)\right).
\end{equation}
In fact, since there is a one-dimensional family of cubic curves passing
through the above $8$ points, we can choose any one that is different
from $C_u$. For example, let
\begin{dmath}
  \label{eq:q5}
    q^{+}(x_1, x_2, z) =  u^{2}(3 u^{10}-15 u^{9}+20 u^{8}+67 u^{7}-371 u^{6}+1024 u^{5}
      -1727 u^{4}+1771 u^{3}\
      -2068 u^{2}+3993 u -3993) x_{1}^{3}-(3 u^{10}-27 u^{9}+100 u^{8}-193 u^{7}+101 u^{6}
      +596 u^{5}-2299 u^{4}+3839 u^{3}-4136 u^{2}+3993 u -3993) x_{1}^{2} x_{2}-4 (4 u^{12}
      -20 u^{11}+31 u^{10}+53 u^{9}-372 u^{8}+1143 u^{7}-1843 u^{6}+1508 u^{5}-1419 u^{4}
      +3135 u^{3}-3036 u^{2}-1331 u +1331) x_{1}^{2} z -4 (3 u^{7}-20 u^{6}+65 u^{5}-118 u^{4}
      +107 u^{3}+143 u^{2}-517 u +517) x_{1} x_{2}^{2}+4 (3 u^{10}-25 u^{9}+90 u^{8}-161 u^{7}
      +33 u^{6}+728 u^{5}-2567 u^{4}+3997 u^{3}-3586 u^{2}+3069 u -3069) x_{1} x_{2} z
      +16(u^{12}-5 u^{11}+8 u^{10}-3 u^{9}-21 u^{8}+176 u^{7}-493 u^{6}+1089 u^{5}
      -3068 u^{4}+5411 u^{3}-5159 u^{2}+1804 u -1804) x_{1} z^{2}+8 (3 u^{7}-17 u^{6}
      +50 u^{5}-72 u^{4}-2 u^{3}+374 u^{2}-891 u +891) x_{2}^{2} z +32 (u^{9}-5 u^{8}+11 u^{7}
      -9 u^{6}-26 u^{5}+196 u^{4}-327 u^{3}+231 u^{2}-55 u +55) x_{2} z^{2}-64 (3 u^{10}
      -21 u^{9}+70 u^{8}-111 u^{7}+47 u^{6}+304 u^{5}-639 u^{4}-367 u^{3}+4554 u^{2}
      -9405 u +9405) z^{3}.
\end{dmath}
Then, the cubic curve $C^{+}$ defined by $q^{+}(x_1, x_2, z)=0$ passes
through the above $8$ points. From there, by taking resultants and
factoring, we obtain the ninth point
$Q^{+}= (x_1(u): x_2(u): z(u))$ on $C_u$, where
\begin{dgroup}\label{eq:Q5}
  \begin{dmath*}
    x_1(u) = -u^{27}+15 u^{26}-111 u^{25}+513 u^{24}-1540 u^{23}+2376 u^{22}+3088 u^{21}-30688 u^{20}
    +86220 u^{19}-59104 u^{18}-525548 u^{17}+2708376 u^{16}-7467922 u^{15}+13121086 u^{14}
    -11661738 u^{13}-9637166 u^{12}+50518468 u^{11}-66858792 u^{10}-29250056 u^{9}
    +251665480 u^{8}-303420084 u^{7}-550150216 u^{6}+3061901612 u^{5}-7015381560 u^{4}
    +10464610827 u^{3}-10737431221 u^{2}+7073843073 u -2357947691,
  \end{dmath*}
  \begin{dmath*}
    x_2(u) = -u^2\left(u^{27}-15 u^{26}+115 u^{25}-573 u^{24}+1980 u^{23}-4448 u^{22}+3668 u^{21}+16836 u^{20}
    -83240 u^{19}+177220 u^{18}-107980 u^{17}-458896 u^{16}+1042758 u^{15}+2513478 u^{14}
    -21451210 u^{13}+70822202 u^{12}-148448124 u^{11}+202614016 u^{10}-137728492 u^{9}
    -58548028 u^{8}+73258240 u^{7}+845573652 u^{6}-3426521076 u^{5}+7372270576 u^{4}
    -10698456879 u^{3}+10815379905 u^{2}-7073843073 u +2357947691\right),
  \end{dmath*}
  \begin{dmath*}
    z(u) = u^4(u^3-11)\left(3 u^{18}-45 u^{17}+345 u^{16}-1746 u^{15}+6373 u^{14}-16869 u^{13}+29401 u^{12}
    -15718 u^{11}-91574 u^{10}+351032 u^{9}-606122 u^{8}+140822 u^{7}+2433673 u^{6}
    -8354445 u^{5}+16882525 u^{4}-23981958 u^{3}+24201573 u^{2}-15944049 u +5314683\right) .
  \end{dmath*}
\end{dgroup}
Replacing $u$ in $Q^{+}$ by $-u$, we have the point
$Q^{-}=(x_1(-u): x_2(-u): z(-u))$ on $C_u$. The two
$\mathbb{Q}(u)$-rational points $Q^{+}$ and $Q^{-}$ on $C_u$ are
converted to the two sections $P^{+}_{\varphi}$ and $P^{-}_{\varphi}$
of $F^{(6)}$ respectively by (\ref{eq:trans}). Computing
$P^{+}_{\varphi} - P^{-}_{\varphi}$ as the group law of elliptic curve
$F^{(6)}$ over $\mathbb{Q}(u)$, and replacing $u^6$ by $s$, we obtain
the section $P_{\varphi}$ of $F^{(1)}$. The coordinates of
$P_{\varphi}$ is given by
\begin{equation}
  \label{eq:P5}
  P_{\varphi} = \left( \frac{f(s)}{192 h(s)^2}, \; \frac{(s-121)g(s)}{512 h(s)^3} \right),
\end{equation}
where
\begin{dgroup}
  \label{eq:fgh5}
  \begin{dmath*}
    f(s) = 3 s^{10}+4242 s^{9}+2430679 s^{8}+730135384 s^{7}+129150804662 s^{6}+16365054527404 s^{5}
    +1890896931056342 s^{4}+156511003892745304 s^{3}+7628511948299823559 s^{2}
    +194918754081273106962 s +2018249984797680027603,
  \end{dmath*}
  \begin{dmath*}
    g(s) = s^{14}+2242 s^{13}+2236395 s^{12}+1318219892 s^{11}+514922124233 s^{10}
    +141266126525854 s^{9}+27916724974734827 s^{8}+3983998405505436120 s^{7}
    +408728770355092602107 s^{6}+30281648805286481009374 s^{5}
    +1616046206494303287179993 s^{4}+60571847938187268807626612 s^{3}
    +1504534724917202541777070395 s^{2}+22083100659160664074343947522 s
    +144209936106499234037676064081,
  \end{dmath*}
  \begin{dmath*}
    h(s) = s(s+121)(5s^2+1958s+73205).
  \end{dmath*}
\end{dgroup}
The height of $P_{\varphi}$ is equal to $10$ by Theorem
\ref{thm:section-F1}, and the isogeny $\varphi$ generates
$\Hom_{\bar{\mathbb{Q}}}(E_1, E_2)$, that is,
$\Hom_{\bar{\mathbb{Q}}}(E_1, E_2) = \langle \varphi \rangle \simeq \langle 5
\rangle$. Thus, we have
\begin{equation}
  F^{(1)}(\bar{\mathbb{Q}}(s)) \simeq \Hom_{\bar{\mathbb{Q}}}(E_1, E_2)\langle 2\rangle
  \simeq \langle 10 \rangle \simeq \langle P_{\varphi} \rangle
\end{equation}
by Proposition \ref{thm:MWL}. It implies that $P_{\varphi}$ forms a basis
of the Mordell-Weil lattice $F^{(1)}(\mathbb{Q}(s))$.

\section{Example: Isogeny of degree $6$}\label{sec:iso6}

Finally, we give an example of the section of $F^{(1)}$ arising
from an isogeny of degree $6$.

Let $E_1$ and $E_2$ be elliptic curves over $k=\mathbb{Q}$ given by
\begin{equation}
  \label{eq:E1E2-6}
  E_1 : y_1^2 = x_1^3+x_1^2-x_1, \quad E_2: y_2^2=x_2^3+x_2^2-36x_2-140.
\end{equation}
The curves $E_1$ and $E_2$ are labeled by 20.a3 and 20.a2 respectively
in LMFDB\cite{LMFDB}. The map
$\varphi: E_1 \to E_2 ; (x_1, x_2) \mapsto (\varphi_x(x_1),
\varphi_y(x_1)y_1)$ is a $6$-isogeny, where
\begin{equation}
  \label{eq:iso6}
  \begin{aligned}
    \varphi_x(x_1) &= \frac{x_1^6+5x_1^4+16x_1^3-5x_1^2-1}{x_1(x_1+1)^2(x_1-1)^2},\\
    \varphi_y(x_1) &= \frac{(x_1^2+1)(x_1^2-4x_1-1)(x_1^4+4x_1^3+6x_1^2-4x_1+1)}{x_1^2(x_1+1)^3(x_1-1)^3}.
  \end{aligned}
\end{equation}
In this case, $C_u$ is given by
\begin{equation}
  \label{eq:Cu6}
  C_u : x_2^3+x_2^2z-36x_2z^2-140z^3 = u^6 (x_1^3+x_1^2z -x_1z^2),
\end{equation}
and it converts to
\begin{equation}
  \label{eq:F6-6}
  F^{(6)} : Y^2 = X^3 -\frac{436}{3}X +\frac{5}{4}u^6 -\frac{18997}{27} - \frac{62500}{u^6},
\end{equation}
with the choice of origin $O=(1:u^2:0)$ by (\ref{eq:trans}). The
elliptic surface $F^{(1)}$ is defined by replacing $u^6$ in
(\ref{eq:F6-6}) by $s$. Substituting $x_2 = \varphi_x(x_1)$ (and
$z=1$) in (\ref{eq:Cu6}), we have
\begin{equation}
  \frac{p^{+}(x_1) p^{-}(x_1) (x_1^2+x_1-1)}{x_1^3(x_1+1)^6(x_1-1)^6}=0,  
\end{equation}
where
\begin{equation}
  \label{eq:p6}
  \begin{aligned}
    p^{+}(x_1) =& (1-u^3)x_1^8 -(10-3u^3)x_1^6-32x_1^5-3u^3x_1^4-32x_1^3+(10+u^3)x_1^2-1,\\
    p^{-}(x_1) =& (1+u^3)x_1^8 -(10+3u^3)x_1^6-32x_1^5+3u^3x_1^4-32x_1^3+(10-u^3)x_1^2-1.
  \end{aligned}
\end{equation}
Applying Theorem \ref{thm:passing} with $d=6$, we obtain the
$1$-dimensional family of cubic curves passing through the $8$ points
$Q^{+}_1=(\alpha_1: \varphi(\alpha_1):1), \ldots, Q^{+}_8=(\alpha_8:
\varphi_x(\alpha_8):1)$, where $\alpha_1, \ldots, \alpha_8$ are $8$
roots of $p^{+}(x_1)=0$. For example, we can take the curve $C^{+}$ in
the family defined by $q^{+}(x_1, x_2, z)=0$, where
\begin{dmath}
  \label{eq:q6}
    q^{+}(x_1, x_2, z) = (u^{9}+2 u^{6}-340 u^{3}+1000) x_{1}^{3}+5 (u^{6}-12 u^{3}-100) x_{1}^{2} x_{2}
    +40 u^{3} (u^{3}-5)  x_{1}^{2}z+12 (u^{3}-10) x_{1} x_{2}^{2}+4 (19 u^{3}+10)  x_{1} x_{2}z
    -(u^{9}-10 u^{6}-340 u^{3}-2600)x_{1}z^2 +4 (u^{3}+10) x_{2}^{2}z
    -(u^{6}-32 u^{3}+180)  x_{2}z^2 -4 (u^{6}+10 u^{3}+300) z^{3}.
\end{dmath}
From there, we obtain the ninth point $Q^{+}=(x_1(u): x_2(u): z(u))$ on $C_u$, where
\begin{dgroup}
  \label{eq:Q6}
  \begin{dmath*}
    x_1(u) =  4 \left(11 u^{6}-380 u^{3}+500\right) \left(u^{12}+4600 u^{6}+250000\right),
  \end{dmath*}
  \begin{dmath*}
    x_2(u) = -4 u^{24}+44 u^{21}+8024 u^{18}-147440 u^{15}-543200 u^{12}+3688000 u^{9}
    -199600000 u^{6}+140000000 u^{3}-7000000000,
  \end{dmath*}
  \begin{dmath*}
    z(u) = \left(u^{8}+6 u^{7}+18 u^{6}+20 u^{5}-40 u^{4}-160 u^{3}-500 u^{2}-1000 u -1000\right)
     \times \left(u^{16}-6 u^{15}+18 u^{14}-68 u^{13}+244 u^{12}-680 u^{11}+1080 u^{10}-960 u^{9}
    +2800 u^{8}-2400 u^{7}-10400 u^{6}+20000 u^{5}+50000 u^{4}-180000 u^{3}
    +500000 u^{2}-1000000 u +1000000\right).
  \end{dmath*}
\end{dgroup}
The point $Q^{-}$ is given by $Q^{-}=(x_1(-u): x_2(-u): z(-u))$. We
can convert $Q^{+}$ and $Q^{-}$ to $P^{+}_{\varphi}$ and
$P^{-}_{\varphi}$ in $F^{(6)}(\mathbb{Q}(u))$ by (\ref{eq:trans}).
Computing $P^{+}_{\varphi} - P^{-}_{\varphi}$ and replacing $u^6$ by
$s$, we obtain $P_{\varphi} \in F^{(1)}(\mathbb{Q}(s))$. It is given by
\begin{equation}
  \label{eq:P6}
  P_{\varphi}=\left( \frac{f(s)}{192 h(s)^2}, \; -\frac{(s^2+50000) g(s)}{512 h(s)^3} \right),
\end{equation}
where
\begin{dgroup}
  \label{eq:fgh6}
  \begin{dmath*}
    f(s) = 3 s^{12}-11568 s^{11}+16300384 s^{10}-9677907200 s^{9}+2291334841600 s^{8}
    -1084577868800000 s^{7}+702031826176000000 s^{6}
    +54228893440000000000 s^{5}+5728337104000000000000 s^{4}
    +1209738400000000000000000 s^{3}+101877400000000000000000000 s^{2}
    +3615000000000000000000000000 s +46875000000000000000000000000,
  \end{dmath*}
  \begin{dmath*}
    g(s) = s^{16}-5784 s^{15}+13675968 s^{14}-16679958400 s^{13}+10338144240640 s^{12}
    -1473647335372800 s^{11}-2426837586892800000 s^{10}
    +1660357937152102400000 s^{9}-221006891984578560000000 s^{8}
    -83017896857605120000000000 s^{7}-6067093967232000000000000000 s^{6}
    +184205916921600000000000000000 s^{5}
    +64613401504000000000000000000000 s^{4}
    +5212487000000000000000000000000000 s^{3}
    +213687000000000000000000000000000000 s^{2}
    +4518750000000000000000000000000000000 s
    +39062500000000000000000000000000000000,
  \end{dmath*}
  \begin{dmath*}
    h(s) = s(37 s^{4}-51760 s^{3}+8556000 s^{2}+2588000000 s +92500000000).
  \end{dmath*}
\end{dgroup}
The isogeny $\varphi$ generates $\Hom(E_1, E_2)$. Thus, the section
$P_{\varphi}$ forms a basis of the Mordell-Weil lattice
$F^{(1)}(\mathbb{Q}(s))$.

\begin{acknowledgments}

  This research was supported by
  Ritsumeikan University Research Promotion Program for Acquiring
  Grants-in-Aid for Scientific Research.
\end{acknowledgments}

\end{document}